\documentclass[a4paper]{article}
\usepackage{amssymb}
\usepackage{amsmath}
\usepackage{amsthm}
\usepackage{hyperref}
\usepackage{enumerate}
\usepackage{mathtools}
\usepackage{url}
\usepackage[T1]{fontenc}

\newtheorem{Theorem} {Theorem} [section]
\newtheorem{Proposition} [Theorem] {Proposition}
\newtheorem{Lemma} [Theorem] {Lemma}

\newtheorem{Fact} [Theorem] {Fact}

\newcommand{\Ff}{{\mathbb F}}

\newcommand{\cC}{{\mathcal C}}

\newcommand{\Aut}{\mathrm{Aut}}

\newcommand{\<}{\langle}
\renewcommand{\>}{\rangle} % was: tabbing command
\renewcommand{\phi}{\varphi} % was: phi

\title{The strongly regular twisted $D_{5,5}(q)$ graph}
\author{
Ferdinand Ihringer
}
\date{11 January 2023}

\begin{document}
\maketitle

\begin{abstract}
  We construct a new family of strongly regular graphs 
  with the same parameters as the strongly regular graphs $D_{5,5}(q)$.
  The construction can be seen as a variant of the construction 
  of twisted Grassmann graphs by Van Dam and Koolen.
\end{abstract}

\section{Introduction}

The {\it Grassmann graph} $J_q(n, k)$ is the graph with the
$k$-spaces of an $n$-dimension\-al vector space over $\Ff_q$ as vertices,
two vertices adjacent if they meet in a $(k-1)$-space.
The construction of the \textit{twisted Grassmann graphs}
by Van Dam and Koolen in 2005 was a breakthrough in the study
of distance-regular graphs.
Twisted Grassmann graphs are the only knownn intransitive family
of distance-regular graphs of unbounded
diameter with the same parameters and
spectrum as the Grassmann graphs $J_q(2k+1, k+1)$
and can be constructed
by modifying $J_q(2k+1, k+1)$, cf.\ \cite{vDK2005}.
It is a classical theme in combinatorics to investigate
if a combiantorial object is determined by some
set of invariants associated with it.
Each distance-regular graph is associated with an
so-called \textit{intersection array}, cf. \cite{BCN}.
The Grassmann graphs $J_q(2k+1, k+1)$ and
the corresponding twisted Grassmann graphs
have the same intersection array, showing that
$J_q(2k+1, k+1)$ is not determined by its
intersection array.

A priori the construction by Van Dam and Koolen naturally generalizes
to other distance-transitive graphs obtained from spherical buildings,
but the necessary conditions are very specific
and some serendipity is involved.
Here we generalize their method to
the strongly regular graphs $D_{5,5}(q)$.
We obtain a new strongly
regular graph with the same parameters, see \S\ref{sec:D5}
definitions. Hence, we obtain the following:

\begin{Theorem}\label{thm:main}
 Let $q$ be a prime power.
 The strongly regular graph $D_{5,5}(q)$ is not determined by
 its parameters.
\end{Theorem}

The strongly regular graph $D_{5,5}(q)$ is a rank $3$ strongly
regular graph, that is, its automorphism group has precisely
three orbits on pairs of vertices. This is a rare phenomenon
and indeed all rank $3$ graphs have been classified, see \cite{BvM}.
Hence, the investigation of $D_{5,5}(q)$ and modifications of it
are of particular interest in the study of strongly regular graphs.
We believe that $D_{5,5}(q)$ and $J(2k+1, k+1)$ are, up to isomorphism,
the only distance-transitive graphs from finite spherical buildings
for which our interpretation of the Van Dam-Koolen method applies.

\section{Twisted Grassmann graphs}

Let us specify what we consider as the two natural generalizations
of the Van Dam-Koolen construction to spherical buildings by giving two
descriptions of the twisted Grassmann graphs.
See \cite{BCN,BvM} for any unexplained notation.
We use algebraic dimensions
together with projective terminology:
we call $1$-spaces {\it points}, $2$-spaces {\it lines},
$3$-spaces {\it planes}, and $4$-spaces {\it solids}.

\subsection{Changing vertices}

If $\Gamma$ is the Grassmann graph $J_q(2k+1, k+1)$,
then the {\it twisted Grassmann graph} $\Gamma'$ can be obtained
as follows: Consider a vector space $V$ of dimension $2k+1$
over $\Ff_q$ with a fixed hyperplane $H$.
The vertices of $\Gamma'$ are the $(k-1)$-spaces of $V$
in $H$ and the $(k+1)$-spaces of $V$
not contained in $H$. Two $(k+1)$-spaces are adjacent
if their meet is a $k$-space, two $(k-1)$-spaces are adjacent
if their meet is a $(k-2)$-space, and a $(k-1)$-space
and a $(k+1)$-space are adjacent if they are incident.

\subsection{Changing edges}

Recall {\it Godsil-McKay (GM) switching}, cf.\ \cite{GM1982}.
Let $\Gamma$ be a graph with vertex set $X$,
and let $\{ C_1, \ldots, C_t, D \}$ be a partition of $X$
such that $\{ C_1, \ldots, C_t\}$ is an equitable partition
of the induced subgraph of $\Gamma$ on $X \setminus D$,
and for every $x \in D$, the vertex $x$ has either
$0$, $\frac12 |C_i|$, or $|C_i|$ neighbors in $C_i$.
Construct a new graph $\Gamma'$ by interchanging adjacency
and nonadjacency between $x \in D$ and $C_i$ whenever
$x$ has $\frac12 |C_i|$ neighbors in $C_i$.
Then $\Gamma$ and $\Gamma'$ are cospectral.

This leads to a second way of obtaining $\Gamma'$ due to
Munemasa, see \cite{Munemasa2015}.
Let $\Gamma$ denote the Grassmann graph $J_q(2k+1, k+1)$.
Let $\sigma$ be a \textit{polarity} of $H$, that is, $\sigma$
is an inclusion-reversing bijection on the subspaces of $H$
such that $\sigma^2$ is the identity.
Let $D$ be the $(k+1)$-spaces of $H$,
for a $k$-space $S$ of $H$ let $C_S$ denote all $(k+1)$-spaces
of $V$ not in $D$ incident with $S$
(or both). Then $\{ C_S: S \text{ is a $k$-space of } H \}$ form an equitable partition
of the induced subgraph on $X \setminus D$.
Moreover, $\{ C_S \cup C_{S^\sigma}: S \text{ is a $k$-space of } H \}$ is still an equitable
partition on the induced subgraph on $X \setminus D$
and we obtain a graph $\Gamma'$
cospectral with $\Gamma$ which is isomorphic
to the twisted Grassmann graph.

\section{Twisted rank 5 hyperbolic spaces} \label{sec:D5}

Let $V$ be a $2m$-dimensional vector space
over $\Ff_q$ provided with a nondegenerate quadratic
form $Q$ with ${-}\det(Q)$ a square in $\Ff_q$,
for instance $Q(x) = x_1x_2 + \ldots + x_{2m-1}x_{2m}$.
We denote the corresponding geometry formed
by the subspaces of $V$ which are totally singular
with respect to $Q$ by $O^+(2m, q)$.
Note that the residue of a point $P$ of $O^+(2m, q)$
is isomorphic to $O^+(2m-2, q)$.
There are two types of totally singular $m$-spaces
with respect to $Q$, arbitrarily labeled \textit{Greeks} and \textit{Latins}.
The meet of two totally singular $m$-spaces of the same type
has even codimension, while the meet of a Greek
and a Latin has odd codimension.

Let $m=5$.
Let $\Gamma$ be the graph with the Greeks of $O^+(10, q)$
as its vertices where two Greeks are adjacent if their meet is a plane,
that is a $3$-space. The graph $\Gamma$ is called $D_{5,5}(q)$.
Recall that a strongly regular graph with parameters $(v, k, \lambda, \mu)$
is a graph (not complete, not edgeless) on $v$ vertices, $k$-regular
such that any pair of adjacent vertices has precisely
$\lambda$ common neighbors, while any pair of distinct
nonadjacent vertices has precisely $\mu$ common neighbors.
It is strongly regular with parameters (see \cite[Th.\ 2.2.20]{BvM})
\begin{align*}
 &v = (q+1)(q^2+1)(q^3+1)(q^4+1), && k = q(q^2+1) \frac{q^5-1}{q-1},\\
 &\lambda = q-1 + q^2(q+1)(q^2+q+1), && \mu = (q^2+1)(q^2+q+1).
\end{align*}
% That is, the graph has $v$ vertices, each vertex has precisely $k$
% neighbors, two adjacent vertices have precisely $\lambda$ neighbors,
% and two nonadjacent vertices have precisely $\mu$ neighbors.

Fix a singular point $P$.
Let $D$ denote the set of Greeks through $P$
and let $C$ denote the set of the remaining Greeks.

\subsection{Changing vertices}

Let $D'$ be the set of lines through $P$.
We define a graph $\Gamma'$ on $C \cup D'$ as follows:
If $x,y \in D'$, then $x,y$ are adjacent if they span
a totally singular plane.
If $x,y \in C$, then $x,y$ are adjacent if they are
adjacent in $\Gamma'$, that is their meet is a plane.
If $x \in D'$ and $y \in C$, then $x,y$ are adjacent
if $x$ meets $y$ in a point.

Let us state the common neighborhood for all pairs of vertices.

\begin{Fact}\label{fact:desc_adj}
  Let $x,y$ be distinct vertices of $\Gamma'$.
  \begin{enumerate}
   \item  If $x,y \in C$ adjacent,
        then $\Gamma'(x) \cap \Gamma'(y)$ consists of
        \begin{enumerate}
          \item all lines $z$ through $P$ such that $z/P$ meets $\< P, P^\perp \cap x\>/P$ and $\< P, P^\perp \cap y\>/P$,
          \item all Greeks not on $P$ which contain the plane $x \cap y$,
          \item all Greeks not on $P$ which meet $x \cap y$ in a line.
        \end{enumerate}
    \item If $y \in C$ and $x \in D'$ with $x,y$ adjacent,
    then $\Gamma'(x) \cap \Gamma'(y)$ consists of
      \begin{enumerate}
        \item all lines through $P$ which meet $P^\perp \cap y$,
        \item all Greeks not on $P$ which meet $\< P, P^\perp \cap y \>$ in a $4$-space,
        \item all Greeks not on $P$ which meet $y$ in a plane through the point $x \cap y$.
      \end{enumerate}
    \item If $x,y \in D'$ adjacent, then $\Gamma'(x) \cap \Gamma'(y)$ consists of
        \begin{enumerate}
          \item all lines $z$ through $P$ such that $\< x, y, z \>$ is totally singular,
          \item all Greeks not on $P$ which meet $\< x, y \>$ in a line.
        \end{enumerate}
   \item  If $x,y \in C$ nonadjacent,
        then $\Gamma'(x) \cap \Gamma'(y)$ consists of
        \begin{enumerate}
          \item all lines $z$ through $P$ such that $z/P$ meets $\< P, P^\perp \cap x\>/P$ and $\< P, P^\perp \cap y\>/P$,
          \item all Greeks not on $P$ which meet $x$ and $y$ each in a plane (necessarily through the point $x \cap y$).
        \end{enumerate}
    \item If $y \in C$ and $x \in D'$ with $x,y$ nonadjacent,
    then $\Gamma'(x) \cap \Gamma'(y)$ consists of
      \begin{enumerate}
        \item all lines through $P$ which meet $x^\perp \cap y$,
        \item all Greeks not on $P$ which meet $x$ in a point and $y$ in a planes of $x^\perp \cap y$.
      \end{enumerate}
    \item If $x,y \in D'$ nonadjacent, then $\Gamma'(x) \cap \Gamma'(y)$ consists of
          all lines $z$ through $P$ such that $\<x,z\>$ and $\<y,z\>$ are totally singular.
%           \hfill \qed
  \end{enumerate}
\end{Fact}

\begin{Lemma}\label{lem:res_p}
  If $x \in C$ and $x' \in D$, then $\< P, P^\perp \cap x\>$ and $x'/P$ are of different type in
  $P^\perp/P \cong O^+(8, q)$. In particular, if $y \in C$, then
  $\< P, P^\perp \cap x\>$ and $\< P, P^\perp \cap y\>$ are the same, meet in a line, or meet trivially.
\end{Lemma}
\begin{proof}
  Clearly, all $x' \in D$ have the same type in $x'/P$ and for $x \in C$,
  we find a Greek $x'$ such that $\< P, P^\perp \cap x\>/P$ and $x'/P$ meet in a plane.
  This shows the first part. Hence, all $x \in C$
  are of the same type in the residue of $P$ which shows the second part.
\end{proof}

\begin{Proposition}
  The graph $\Gamma'$ is strongly regular
  with the same parameters as $\Gamma$.
\end{Proposition}
\begin{proof}
  The number of singular points in $O^+(8, q)$
  equals the number of Greeks in $O^+(8, q)$.
  Hence, $\Gamma$ and $\Gamma'$ have the same number of vertices.

  The graph is $k$-regular:
  If $x \in D'$, then we see in the quotient of $P$ that
  $x$ is coplanar with $q(q^2+1)(q^2+q+1)$
  lines in $D'$. A point lies on $(q+1)(q^2+1)(q^3+1)$ Greeks,
  so $x$ meets $q \cdot q^3(q+1)(q^2+1)$ Greeks not through $P$
  in a point.
  If $x \in C$, then $x$ is adjacent to $q \frac{(q^5-1)(q^4-1)}{(q^2-1)(q-1)} - \frac{q^4-1}{q-1}$ elements
  in $C$ and meets $\frac{q^4-1}{q-1}$ lines through $P$.
  Hence, $\Gamma'$ has degree $k$.

  For $x,y$ adjacent, we find $|\Gamma'(x) \cap \Gamma'(y)| = \lambda$
  using Fact \ref{fact:desc_adj}. More precisely:
  If $x,y \in C$ and $P^\perp \cap x \cap y$ is a line,
  then in $\Gamma'(x) \cap \Gamma'(y)$
  we find $q+1$ lines as in Fact \ref{fact:desc_adj}.1a,
  $q-1$ Greeks as in Fact \ref{fact:desc_adj}.1b,
  and $(q^2+q)\cdot q^2(q+1) + (q^2-1)(q+1)$ Greeks
  as in Fact \ref{fact:desc_adj}.1c.
  If $x,y \in C$ and $P^\perp \cap x \cap y$ is a plane,
  then in $\Gamma'(x) \cap \Gamma'(y)$ we find $q^3+q^2+q+1$ lines
  as in Fact \ref{fact:desc_adj}.1a (using Lemma \ref{lem:res_p}),
  $q-2$ Greeks as in Fact \ref{fact:desc_adj}.1b,
  and $(q^2+q+1) \cdot ((q^2+1)(q+1) - 2q-1)$ Greeks as in Fact \ref{fact:desc_adj}.1c.
  If $x \in D'$ and $y \in C$, then in $\Gamma'(x) \cap \Gamma'(y)$
  we find $q^3+q^2+q$ lines as in Fact \ref{fact:desc_adj}.2a,
  $q^4-1$ Greeks which meet $\<P, P^\perp \cap y\>$
  in a $4$-space as in Fact \ref{fact:desc_adj}.1b, and $q^2(q^2+q+1) \cdot q$
  Greeks which meet $y$ in a plane through $x \cap y$ as in Fact \ref{fact:desc_adj}.1c.
  If $x,y \in D'$, then in $\Gamma'(x) \cap \Gamma'(y)$
  we find $q-1 + q^2(q+1)^2$ lines as in Fact \ref{fact:desc_adj}.3a,
  and $q^4(q+1)$ Greeks as in Fact \ref{fact:desc_adj}.3b.

  For $x,y$ nonadjacent, we find $|\Gamma'(x) \cap \Gamma'(y)| = \mu$
  using Fact \ref{fact:desc_adj}. More precisely:
  If $x,y \in C$ and $P^\perp \cap x \cap y$
  is empty, then $\Gamma'(x) \cap \Gamma'(y) = \Gamma(x) \cap \Gamma(y)$
  as in Fact \ref{fact:desc_adj}.4 (using Lemma \ref{lem:res_p}).
  If $x,y \in C$ and $P^\perp \cap x \cap y$ is a point,
  then in $\Gamma'(x) \cap \Gamma'(y)$ we find $q+1$ lines as in Fact \ref{fact:desc_adj}.4a (using Lemma \ref{lem:res_p}),
  and $(q^2+q+1)(q^2+1) - (q+1)$ Greeks as in Fact \ref{fact:desc_adj}.4b.
  If $x \in D'$ and $y \in C$, then
  in $\Gamma'(x) \cap \Gamma'(y)$ we find $q^2+q+1$ lines as in Fact \ref{fact:desc_adj}.5a,
  and $q \cdot (q^3+q^2+q)$ Greeks as in Fact \ref{fact:desc_adj}.5b.
  If $x,y \in D'$, then $\Gamma'(x) \cap \Gamma'(y) \subseteq D'$,
  so by triality (or duality) and Fact \ref{fact:desc_adj}.6, $|\Gamma'(x) \cap \Gamma'(y)| = |\Gamma(x') \cap \Gamma(y')|$
  for $x',y' \in D$ nonadjacent in $\Gamma$.
\end{proof}

The maximal cliques of $\Gamma$
have size $q^3+q^2+q+1$ or $q^4+q^3+q^2+q+1$, see \cite[Proposition 3.2.5.]{BvM}.
These correspond to all Greeks through a fixed line, respectively,
all Greeks which meet a fixed solid in at least a plane.

\begin{Lemma}\label{lem:cliques_D55}
  Let $\cC$ and $\cC'$ be distinct maximal cliques of $\Gamma$.
  Then $|\cC \cap \cC'| \leq q^2+q+1$.
\end{Lemma}
\begin{proof}
  If $\cC$, respectively, $\cC'$ is the set of all Greeks through
  a line $L$, respectively, $L'$, then $\cC \cap \cC'$
  is the set of all Greeks through $\< L, L'\>$ and
  $\cC \cap \cC'$ is largest if $\< L, L'\>$ is a totally singular plane.
  Then $\cC \cap \cC'$ consists of the $q+1$ Greeks through $\< L, L'\>$.
  If $\cC'$ instead is the set of all Greeks which meet a fixed solid $S$
  in a plane of solid, then $\cC \cap \cC'$ is largest if
  $L \subseteq S$. Then $|\cC \cap \cC'| = q^2+q+1$.
  If $\cC$, respectively, $\cC'$ is the set of all Greeks which
  meet a solid $S$, respectively, $S'$ in at least a plane, then $\cC \cap \cC'$
  is largest if $S \cap S'$ is a plane. Then $|\cC \cap \cC'| = q+1$.
\end{proof}

Next we apply Lemma \ref{lem:cliques_D55} to classify all maximal cliques in $\Gamma'$.

\begin{Lemma}\label{lem:max_D5_cliques}
  Let $\cC$ be a clique of $\Gamma'$.
  Call $(a,b)$ the type of $\cC$ where
  $a = |\cC \cap D'|$ and $b = |\cC \cap C|$.
  If $\cC$ is maximal, then one of the following occurs:
  \begin{enumerate}
   % all through line L not in p^perp plus line through p and p^\perp \cap L
   \item $|\cC| = q^3+q^2+q+2$ of type $(1, q^3+q^2+q+1)$,
   where $\cC \cap C$ consists of the $(q+1)(q^2+1)$ Greeks
   through a line $L$ not in $P^\perp$ and
   $\cC \cap D' = \{ \< P, P^\perp \cap L\>\}$.
   \item $|\cC| = q^3+q^2+q+1$ of type $(q+1,q^3+q^2)$,
   where $\cC \cap C$ consists of the $q^2(q+1)$ Greeks
   through a line $L$ in $P^\perp$ and
   $\cC \cap D'$ consists of the $q+1$ lines through $P$ which meet $L$.
   \item $|\cC| = q^3+q^2+q+1$ of type $(q^3+q^2+q+1,0)$,
   where $\cC$ consists of all lines through $P$ in a fixed Greek.
   \item $|\cC| = q^4+q^3+q^2+q$ of type $(0,q^4+q^3+q^2+q)$,
   where $\cC$ consists of all Greeks not through $P$ which meet
   a solid $S$ not in $P^\perp$ in a plane.
   \item $|\cC| = q^4+q^3+q^2+q$ of type $(0,q^4+q^3+q^2+q)$,
   where $\cC$ consists of all Greeks which
   meet a fixed solid $S$, where $\<P,S\>$ is a Greek, in a plane.
   \item $|\cC| = q^4+q^3+q^2+q+1$ of type $(q^3+q^2+q+1,q^4)$,
   where $\cC \cap C$ consists of all Greeks disjoint from $P$
   which meet a fixed Latin $L$ through $P$ in a hyperplane,
   and $\cC \cap D'$ consists of the lines through $P$ in $L$.
   \item $|\cC| = q^4+q^2+q+1$ of type $(q^2+q+1,q^4)$,
   where $\cC \cap C$ consists of the $q^4$ Greeks disjoint from $P$
   which meet a fixed solid $S$ through $P$ in a plane,
   and $\cC \cap D'$ consists of the $q^2+q+1$ lines through $P$ in $S$.
%    \hfill \qed
  \end{enumerate}
\end{Lemma}
\begin{proof}
  We have that $\cC \cap C$ lies in one of the two types
  of maximal cliques of $\Gamma$, either all Greeks through a fixed line
  or all Greeks which meet a fixed solid in at least a plane.
  Furthermore, $\Gamma'$ restricted to $D'$ is isomorphic to $O^+(8, q)$.
  The maximal cliques of $O^+(8, q)$ all have size $q^3+q^2+q+1$
  and correspond to the totally singular $4$-spaces of $O^+(8, q)$.
  Hence, $\cC \cap D'$ is contained in a totally singular $5$-space.

  Suppose that $|\cC \cap C| > q^2+q+1$.
  By Lemma \ref{lem:cliques_D55}, the maximal clique, restricted to $C$,
  lies in a unique maximal clique of $\Gamma$.
  If $\cC \cap C$ is contained in a maximal clique of $\Gamma$
  which consists of all Greeks which contain a fixed line $L$,
  then we find, depending on $P^\perp \cap L$ a point
  or $P^\perp \cap L = L$, cases (1) and (2).
  If $\cC \cap C$ is contained in a maximal clique of $\Gamma$
  which consists of all Greeks which meet a fixed solid $S$ in a plane
  or a solid. There are five cases: $P^\perp \cap S$ is a plane,
  $\< P, S \>$ is a Greek, $\< P, S \>$ is a Latin, or $P \subseteq S$.
  This yields (4)--(7).

%   We saw in the above that a maximal clique of
%   $\Gamma$ intersects $D$ in at least $q^3+q^2$
%   elements unless it consists of all Greeks
%   through a fixed line $L$ on $P$.
  Now suppose that $|\cC \cap C| \leq q^2+q+1$.
  We saw in the above that this can only occur if $\cC \cap D'$ spans
  a Greek which corresponds to case (3).
\end{proof}

Thus, $\Gamma$ and $\Gamma'$ are not isomorphic
which shows Theorem \ref{thm:main}.
By \cite[Th.\ 9.4.8]{BCN}, $\Aut(\Gamma) \cong P\Gamma{}\Omega{}^+_{10}(q)$,
an index two subgroup of $P\Gamma{}O^+_{10}(q)$.
It is clear that $\Aut(\Gamma')$ contains the stabilizer
of $P$ in $P\Gamma{}\Omega{}^+_{10}(q)$.
An element of $\Aut(\Gamma')$ maps a maximal clique
to a maximal clique of the same size. By Lemma \ref{lem:max_D5_cliques},
maximal cliques of size $q^4+q^3+q^2+q$ contain no vertices of $D'$.
Hence, $D'$ and $C$ are each orbits of $\Aut(\Gamma')$.
As Lemma \ref{lem:max_D5_cliques} provides the totally singular lines
not on $P$, one can indeed reconstruct the whole geometry, we see
that $\Aut(\Gamma')$ is a subgroup of $\Aut(\Gamma)$.

% \begin{Proposition}
%   The automorphism group of $\Aut(\Gamma')$ is the stabilizer of $P$ in $P\Gamma{}\Omega{}^+_{10}(q)$. \hfill \qedhere
% \end{Proposition}
% \begin{proof}
%   Let $\phi$ be an element of $\Aut(\Gamma')$.
%   Identify totally singular subspaces of $Q$ not on $P$ with
%   all the Greeks in $D$ incident with them.
%   By Lemma \ref{lem:max_D5_cliques}, $\phi$ maps
%   totally singular lines not on $P$ to totally singular lines not on $P$.
%   Hence, $\phi$ induces a collineation $h$ on the totally singular lines not on $P$.
%
%   It remains to show that all elements $h' \in \Aut(\Gamma')$,
%   which $h$ extends to, are a collineation. For any point $R$ collinear with $P$,
%   we find totally singular lines $L_1$ and $L_2$ such that $L_1 \cap L_2 = P$
%   and $P^\perp \cap L_1 = P^\perp \cap L_2 = R$.
%   Note that $L_1$ and $L_2$ correspond to cliques as in Lemma \ref{lem:max_D5_cliques}.1.
%   Then there exists a unique element $x$ of $C'$ which contains $R$.
%   Hence, a automorphism $h'$ of $\Aut(\Gamma')$ maps $x$ to $\< P, (L_1 \cap L_2)^h\>$.
%   Thus, $h'$ is a collineation.
% \end{proof}

A graph satisfies the {\it $t$-vertex condition}
if, for all triples $(T, x_0, y_0)$ consisting
of a $t$-vertex graph $T$ together with two distinct
distinguished vertices $x,y$ of $\Gamma$, the number of
copies of $T$ in $\Gamma$, where the isomorphism maps $x_0$
to $x$ and $y_0$ to $y$, does not depend on the choice
of the pair $x, y$ but only on whether $x, y$
are adjacent or nonadjacent.

Strongly regular graphs are precisely the graphs
satisfying the $3$-vertex condition.
Rank $3$ graphs are precisely the
graphs satisfying the $v$-vertex condition.
Excluding rank $3$ graphs, not many graphs
satisfying the $4$-vertex condition are known, see \cite{BIK2023} for a recent survey.
Sims observed that a strongly regular graph
satisfies the $4$-vertex condition with parameters $(\alpha, \beta)$
if and only if the number of edges in $\Gamma(x) \cap \Gamma(y)$ is $\alpha$ (resp.\ $\beta$)
whenever the vertices $x, y$ are adjacent (resp.\ nonadjacent).
Testing this (either by hand or computer) we find

\begin{Proposition}
  For $q=2$, the graph $\Gamma'$ satisfies the $4$-vertex condition with
  parameters $(1554, 315)$. \hfill \qed
\end{Proposition}

For $q \geq 3$, no $\Gamma'$ satisfies the $4$-vertex condition.

\subsection{Changing edges}

Alternatively, we can obtain a graph isomorphic
to $\Gamma'$ by considering the following
equitable partition: We have that the residue of $P$
possesses a polarity $\sigma$ which interchanges singular points and Latins.
Let $C_S$ denote the set of Greeks not on $P$
which meet a Latin $S$ on $P$ in a 4-space.
Then $\{ C_S: S \text{ is a Latin on $P$} \}$ is an equitable partition of the induced
subgraph on $X \setminus D$.
Moreover, $\{ C_S \cup C_{S^\sigma}: S \text{ is a Latin on $P$} \}$ is still an equitable partition of $X \setminus D$
and satisfies the conditions of GM switching
as each element of $D$ either contains 0, 1, or 2 elements of $\{ S, S^\sigma \}$.
Hence, we obtain a graph $\Gamma'$ cospectral with $\Gamma$.
As $\Gamma$ is strongly regular, $\Gamma'$ is strongly regular with
the same parameters.

\bigskip
\paragraph*{Acknowledgments}
The author thanks Andries E.\ Brouwer and Hendrik Van Maldeghem for various remarks.
The author is supported by a
postdoctoral fellowship of the Research Foundation -- Flanders (FWO).

%% The Appendices part is started with the command \appendix;
%% appendix sections are then done as normal sections
%% \appendix

%% \section{}
%% \label{}

%% If you have bibdatabase file and want bibtex to generate the
%% bibitems, please use
%%
%%  \bibliographystyle{elsarticle-num}
%%  \bibliography{<your bibdatabase>}

%% else use the following coding to input the bibitems directly in the
%% TeX file.

\end{document}